\documentclass[11pt,reqno]{amsart}
\usepackage[utf8]{inputenc}
\usepackage{amsmath,amssymb,mathtools}
\usepackage[margin=1.05in]{geometry}
\usepackage{microtype}
\usepackage{enumitem}
\usepackage[hidelinks]{hyperref}
\hypersetup{
 pdftitle={Central comparison and adjoint cohomology of seaweed Lie algebras},
 pdfauthor={Vincent E. Coll, Jr. and Alan Hylton},
 pdfsubject={Lie algebra cohomology of seaweed subalgebras},
 pdfkeywords={seaweed Lie algebra, biparabolic, cohomology, rigidity, deformation}
}
\newtheorem{theorem}{Theorem}[section]
\newtheorem{lemma}[theorem]{Lemma}
\newtheorem{proposition}[theorem]{Proposition}
\newtheorem{corollary}[theorem]{Corollary}
\theoremstyle{definition}
\newtheorem{definition}[theorem]{Definition}
\newtheorem{example}[theorem]{Example}
\theoremstyle{remark}
\newtheorem{remark}[theorem]{Remark}
\newcommand{\C}{\mathbb C}
\newcommand{\g}{\mathfrak g}
\newcommand{\s}{\mathfrak s}
\newcommand{\h}{\mathfrak h}
\newcommand{\rr}{\mathfrak r}
\newcommand{\n}{\mathfrak n}
\newcommand{\p}{\mathfrak p}
\newcommand{\uu}{\mathfrak u}
\newcommand{\llie}{\mathfrak l}
\newcommand{\qq}{\mathfrak q}
\newcommand{\kk}{\mathfrak k}
\newcommand{\cc}{\mathfrak c}
\DeclareMathOperator{\ad}{ad}
\DeclareMathOperator{\Der}{Der}
\DeclareMathOperator{\Out}{Out}
\DeclareMathOperator{\Hom}{Hom}
\DeclareMathOperator{\Nil}{Nil}
\DeclareMathOperator{\spn}{span}
\title[Adjoint cohomology of seaweed Lie algebras]{Central comparison and adjoint cohomology of seaweed Lie algebras}
\author{Vincent E. Coll, Jr.}
\thanks{Corresponding author.}
\address{Department of Mathematics, Lehigh University, Bethlehem, PA 18015, USA}
\email{vec208@lehigh.edu}
\author{Alan Hylton}
\thanks{Much of the work leading to this paper was carried out in the second author's 2022 Ph.D. dissertation, \emph{On the Cohomology of Seaweed Algebras}, at Lehigh University under the direction of the first author \cite{Hyl22}.}
\address{NASA Goddard Space Flight Center, Greenbelt, MD 20771, USA}
\date{}

\begin{document}
\begin{abstract}
Let $\s$ be a seaweed subalgebra of a complex simple Lie algebra, let $\rr$ be its standard
reductive factor, and let $Z=Z(\s)$. We prove that the coefficient inclusion $Z\hookrightarrow\s$
induces a quasi-isomorphism
\[
 C^\bullet(\s,Z)\longrightarrow C^\bullet(\s,\s).
\]
We also prove that inflation induces an isomorphism $H^\bullet(\rr,\C)\cong H^\bullet(\s,\C)$.
Consequently,
\[
 H^\bullet(\s,\s)\cong H^\bullet(\rr,\C)\otimes Z
\]
canonically as a graded $H^\bullet(\s,\C)$-module. The proof decomposes the invariant cochain
complex of the non-Levi nilpotent ideal into two one-sided complexes attached to parabolic
subalgebras and applies Tolpygo's adjoint-cohomology vanishing theorem.

It follows that a seaweed is adjoint-acyclic precisely when it is indecomposable, and this is also
equivalent to completeness. In particular, every Frobenius seaweed is complete and adjoint-acyclic.
We obtain explicit formulas for outer derivations, infinitesimal deformations, obstruction spaces,
and all adjoint Betti numbers in terms of the center and the exponents of the reductive factor.
The full shifted central-valued cochains form a quasi-isomorphic differential graded Lie subalgebra
of the full shifted deformation complex. The formulas are uniform across all finite Dynkin types and
specialize to the classical parabolic and reductive cases.
\end{abstract}
\maketitle

\section{Introduction}
A seaweed, or biparabolic, subalgebra is the intersection of two parabolic subalgebras whose
sum is the ambient reductive Lie algebra. Seaweeds contain parabolic and Levi subalgebras as
limiting cases and form an important testing ground for index theory, coadjoint invariant theory,
and explicit combinatorics; see, for example, \cite{DK00,Pan01}. Their deformation theory is governed
by adjoint Lie algebra cohomology: $H^2(\s,\s)$ is the space of infinitesimal deformations modulo
equivalence, and obstruction classes occur in $H^3(\s,\s)$ \cite{NR67}.

The main result is not only a computation of these two groups, but a comparison theorem in
every degree. Throughout the paper, all Lie algebras and modules are finite-dimensional over $\C$.

\begin{theorem}[Central comparison]\label{thm:main}
Let $\s$ be a seaweed subalgebra of a complex simple Lie algebra.
Let $\rr$ be its standard reductive factor, let $\n$ be the non-Levi nilpotent ideal defined
in \eqref{eq:n}, and put $Z=Z(\s)$. Then:
\begin{enumerate}[label=(\roman*),leftmargin=*,itemsep=3pt]
\item the coefficient inclusion $Z\hookrightarrow\s$ induces an isomorphism
\[
 H^n(\s,Z)\xrightarrow{\ \sim\ }H^n(\s,\s)\qquad(n\geq0);
\]
\item inflation along the quotient $\s\twoheadrightarrow\s/\n\cong\rr$ induces an isomorphism of graded algebras
\[
 H^\bullet(\rr,\C)\xrightarrow{\ \sim\ }H^\bullet(\s,\C);
\]
\item consequently, there is a canonical isomorphism of graded $H^\bullet(\s,\C)$-modules
\[
 H^\bullet(\s,\s)\cong H^\bullet(\rr,\C)\otimes Z.
\]
\end{enumerate}
\end{theorem}

Thus the ideal $\n$ is invisible both to ordinary cohomology and to the noncentral part of
adjoint cohomology. The theorem specializes at the two extremes as follows. For a parabolic
subalgebra the center is zero, and one recovers Tolpygo's vanishing theorem \cite{Tol72}. For a Levi
subalgebra $\llie$, it reduces to the classical formula
\[
 H^\bullet(\llie,\llie)\cong H^\bullet(\llie,\C)\otimes Z(\llie).
\]
For biparabolic subalgebras of $\mathfrak{sl}_n$, Elashvili and Rakviashvili stated the
all-degree coefficient comparison
$H^n(\s,\s)\cong H^n(\s,Z(\s))$ and sketched an argument based on a Tolpygo-type
homotopy and the Hochschild--Serre spectral sequence \cite{ER16}. They also suggested that an
analogous result should hold for biparabolics in arbitrary semisimple Lie algebras.
Theorem~\ref{thm:main} establishes this comparison uniformly in every finite type, identifies
the natural coefficient inclusion as a quasi-isomorphism, and determines the ordinary cohomology
and deformation-complex consequences.

The central comparison has several structural consequences. A standard seaweed $\p(A\mid B)$ is
called indecomposable when $A\cup B$ is the full set of simple roots. We prove
\[
 \s\text{ indecomposable}\iff Z(\s)=0\iff H^\bullet(\s,\s)=0\iff\s\text{ complete}.
\]
Since the center of a Frobenius Lie algebra is zero, every Frobenius seaweed is therefore complete
and adjoint-acyclic.

The result also gives closed formulas. If $\cc=Z(\rr)$ and $[\rr,\rr]=\bigoplus_a\rr_a$ is the decomposition into
simple ideals, then
\[
 H^1(\s,\s)\cong\cc^*\otimes Z,\qquad
 H^2(\s,\s)\cong\bigwedge^2\cc^*\otimes Z,
\]
and
\[
 H^3(\s,\s)\cong\left(\bigwedge^3\cc^*\oplus\C^m\right)\otimes Z,
\]
where $m$ is the number of simple ideals $\rr_a$. More generally, all adjoint Betti numbers are read from
the exponents of these simple ideals.

The proof has two ingredients. Write $\s=\rr\ltimes\n$. The ideal $\n$ is the direct sum of two commuting
one-sided nilpotent ideals. Weight considerations show that, in positive degree, every $\rr$-invariant cochain
is supported on one side or the other. Each one-sided invariant complex is a relative complex for a
parabolic subalgebra, up to a harmless central summand in degree zero. Tolpygo's theorem makes
both one-sided complexes acyclic in positive degrees. A short degree-one gluing argument then gives
\[
 H^q(\n,\s)^{\rr}=\begin{cases}Z(\s),&q=0,\\0,&q>0.\end{cases}
\]
The Hochschild--Serre spectral sequence and its comparison map finish the proof.

Section~2 records the root-theoretic structure, the center, and the central splitting. Section~3
proves the invariant-cohomology theorem for $\n$. Section~4 proves Theorem~\ref{thm:main}.
Section~5 develops structural, cohomological, and deformation-theoretic consequences, and
Section~6 gives examples.

\section{Root-theoretic structure and the center}
Let $\g$ be a complex simple Lie algebra with Cartan subalgebra $\h$ and triangular decomposition
\[
 \g=\uu^-\oplus\h\oplus\uu^+.
\]
Let $\Delta$ be the corresponding root system, $\Delta^+$ the positive roots, and
$\Pi=\{\alpha_1,\ldots,\alpha_\ell\}$ the simple roots. For $I\subseteq\Pi$, put
\[
 \Delta_I=\Delta\cap\mathbb Z I,\qquad \Delta_I^+=\Delta_I\cap\Delta^+.
\]
The standard and opposite standard parabolics attached to $I$ are
\begin{align*}
 \p_I&=\h\oplus\bigoplus_{\alpha\in\Delta^+}\g_\alpha
             \oplus\bigoplus_{\alpha\in\Delta_I^+}\g_{-\alpha},\\
 \p_I^-&=\h\oplus\bigoplus_{\alpha\in\Delta^-}\g_\alpha
             \oplus\bigoplus_{\alpha\in\Delta_I^+}\g_\alpha.
\end{align*}
Every seaweed is conjugate to one in the following standard form \cite{Pan01}.

\begin{definition}\label{def:seaweed}
For $A,B\subseteq\Pi$, the associated standard seaweed is
\[
 \s=\p(A\mid B):=\p_A\cap\p_B^-.
\]
We call it \emph{indecomposable} when $A\cup B=\Pi$.
\end{definition}
Set
\[
 C=A\cap B,\qquad S=A\cup B.
\]
The root-space description is
\begin{equation}\label{eq:roots}
 \s=\h\oplus\bigoplus_{\alpha\in\Delta_B^+}\g_\alpha
           \oplus\bigoplus_{\alpha\in\Delta_A^+}\g_{-\alpha}.
\end{equation}
Its standard reductive factor is
\begin{equation}\label{eq:r}
 \rr=\h\oplus\bigoplus_{\alpha\in\Delta_C}\g_\alpha.
\end{equation}
The complementary nilpotent ideal formed by the non-Levi root spaces is
\begin{equation}\label{eq:n}
 \n=\uu_A^-\oplus\uu_B^+,
\end{equation}
where
\[
 \uu_A^-=\bigoplus_{\alpha\in\Delta_A^+\setminus\Delta_C^+}\g_{-\alpha},
 \qquad
 \uu_B^+=\bigoplus_{\beta\in\Delta_B^+\setminus\Delta_C^+}\g_\beta.
\]
Thus $\s=\rr\ltimes\n$ and $\s/\n\cong\rr$. The ideal $\n$ is the Lie algebra of the
unipotent radical; it need not be the full Lie-algebra nilradical.

\begin{lemma}\label{lem:commute}
The two one-sided nilpotent ideals commute:
\[
 [\uu_A^-,\uu_B^+]=0.
\]
\end{lemma}
\begin{proof}
Take $\alpha\in\Delta_A^+\setminus\Delta_C^+$ and
$\beta\in\Delta_B^+\setminus\Delta_C^+$. The simple-root expansion of $-\alpha+\beta$ has a negative
coefficient at some root of $A\setminus B$ and a positive coefficient at some root of $B\setminus A$.
A root has either all nonnegative or all nonpositive simple-root coefficients. Therefore $-\alpha+\beta$
is neither a root nor zero, and the corresponding root spaces commute.
\end{proof}
Let $\{\varpi_i^\vee\}_{i=1}^\ell$ be the basis of $\h$ dual to the simple roots, so that
$\alpha_j(\varpi_i^\vee)=\delta_{ij}$.

\begin{theorem}\label{thm:center}
For $\s=\p(A\mid B)$,
\[
 Z(\s)=\{h\in\h:\alpha(h)=0\text{ for every }\alpha\in S\}
      =\bigoplus_{\alpha_i\notin S}\C\varpi_i^\vee.
\]
In particular,
\[
 \dim Z(\s)=\ell-|S|,
\]
and $\s$ is indecomposable if and only if its center is zero.
\end{theorem}
\begin{proof}
Since $\h\subseteq\s$, a central element cannot have a nonzero root-space component. Hence $Z(\s)\subseteq\h$.
For every $\alpha_i\in A$, the root space $\g_{-\alpha_i}$ lies in $\s$, and for every $\alpha_i\in B$,
the root space $\g_{\alpha_i}$ lies in $\s$. Thus a central $h\in\h$ annihilates every simple root in $S$.

Conversely, every root occurring in \eqref{eq:roots} is supported on $A$ or on $B$. An element of $\h$ that
annihilates all simple roots in $S$ therefore centralizes every root space of $\s$. The stated basis and
dimension follow.
\end{proof}

For clarity, if $\Nil(\s)$ denotes the largest nilpotent ideal of the Lie algebra, then
\[
 \Nil(\s)=Z(\s)\oplus\n.
\]
Indeed, $Z(\s)\oplus\n$ is a nilpotent ideal. Conversely, the image of $\Nil(\s)$ in
$\s/\n\cong\rr$ is a nilpotent ideal of the reductive algebra $\rr$, hence lies in $Z(\rr)$.
Since $\n\subseteq\Nil(\s)$, it follows that
$\Nil(\s)=(\Nil(\s)\cap Z(\rr))\oplus\n$.
For $h\in\Nil(\s)\cap Z(\rr)$, the action $\ad h$ on $\n$ is both diagonalizable and nilpotent,
so it is zero. Such an $h$ centralizes both $\rr$ and $\n$, and therefore belongs to $Z(\s)$.
The quotient used throughout this paper is $\s/\n$, not $\s/\Nil(\s)$.

The center splits off as a Lie algebra direct factor. For $I\subseteq\Pi$, let
\[
 \h_I=\spn\{h_{\alpha_i}:\alpha_i\in I\},
\]
where $h_{\alpha_i}$ are the simple coroots.

\begin{proposition}[Central splitting]\label{prop:split}
For $S=A\cup B$ one has
\[
 \h=Z(\s)\oplus\h_S.
\]
Moreover,
\[
 \s_S:=\h_S\oplus\bigoplus_{\g_\alpha\subseteq\s}\g_\alpha
\]
is an ideal of $\s$, and
\[
 \s=Z(\s)\oplus\s_S
\]
as a direct sum of Lie algebras. In particular, $\s_S\cong\s/Z(\s)$.
\end{proposition}
\begin{proof}
The principal finite-type Cartan matrix indexed by $S$ is nonsingular. Hence an element of
$\h_S$ annihilated by every simple root in $S$ is zero. Theorem~\ref{thm:center} gives
$\h_S\cap Z(\s)=0$, and dimensions give the first direct sum.

The root set in \eqref{eq:roots} is closed under addition whenever the sum is a root. Thus brackets of
nonopposite root spaces in $\s_S$ remain in $\s_S$. If both $\g_\alpha$ and $\g_{-\alpha}$ occur, then
$\alpha\in\Delta_C$, so their bracket lies in $\h_C\subseteq\h_S$. Therefore $\s_S$ is a subalgebra.
Every root of $\s$ is supported on $S$, so $Z(\s)$ centralizes $\s_S$. Both summands are consequently ideals.
\end{proof}

\begin{proposition}\label{prop:components}
The quotient $\s/Z(\s)$ is a direct sum of indecomposable seaweed algebras.
\end{proposition}
\begin{proof}
Let $\Gamma_1,\ldots,\Gamma_m$ be the connected components of the Dynkin subdiagram on $S$. The
semisimple regular subalgebra generated by the roots supported on $S$ decomposes as
\[
 \g_S=\bigoplus_{a=1}^m\g_{\Gamma_a},\qquad
 \h_S=\bigoplus_{a=1}^m\h_{\Gamma_a}.
\]
A root has connected support, so Proposition~\ref{prop:split} gives
\[
 \s/Z(\s)\cong\bigoplus_{a=1}^m
 \p_{\g_{\Gamma_a}}\bigl(A\cap\Gamma_a\mid B\cap\Gamma_a\bigr).
\]
For every $a$, $(A\cap\Gamma_a)\cup(B\cap\Gamma_a)=\Gamma_a$, so each summand is indecomposable.
\end{proof}

\section{Invariant cohomology of the nilpotent ideal}
We compute the invariant cohomology entering the Hochschild--Serre spectral sequence for
$\s=\rr\ltimes\n$. The adjoint action of $\rr$ on $\s$ and $\n$ is semisimple: the derived algebra of
$\rr$ acts completely reducibly, and $Z(\rr)$ acts diagonally by root weights. Hence every cochain space
$C^q(\n,\s)$ and each of its subquotients is a semisimple $\rr$-module. Taking $\rr$-invariants is therefore exact, and
\begin{equation}\label{eq:invariants}
 H^q\bigl(C^\bullet(\n,\s)^{\rr}\bigr)\cong H^q(\n,\s)^{\rr}.
\end{equation}
We use the following standard consequence of Hochschild--Serre theory \cite{HS53}.

\begin{lemma}[Hochschild]\label{lem:hoch}
Let $\llie$ be reductive and let $V$ be a finite-dimensional semisimple $\llie$-module. Then
\[
 H^p(\llie,V)\cong H^p(\llie,\C)\otimes V^{\llie}\qquad(p\geq0).
\]
In particular, if $V^{\llie}=0$, then $H^\bullet(\llie,V)=0$.
\end{lemma}

\subsection{The parabolic input}
Tolpygo proved that the adjoint cohomology of a parabolic subalgebra
of a complex semisimple Lie algebra vanishes in every degree \cite{Tol72}. We require the following
relative consequence.

\begin{lemma}\label{lem:relative}
Let $\p=\llie\ltimes\uu$ be a parabolic subalgebra of a complex semisimple Lie algebra, with
Levi factor $\llie$ and nilradical $\uu$. Then
\[
 H^q(\uu,\p)^{\llie}=0\qquad(q\geq0).
\]
\end{lemma}
\begin{proof}
The $\llie$-module $H^q(\uu,\p)$ is semisimple, because it is a subquotient of a semisimple cochain
module. In degree zero,
\[
 H^0(\uu,\p)^{\llie}=Z(\p)=0.
\]
Suppose the assertion fails and choose the least $q>0$ for which the invariant space is nonzero. In
the Hochschild--Serre spectral sequence
\[
 E_2^{a,b}=H^a\bigl(\llie,H^b(\uu,\p)\bigr)\Longrightarrow H^{a+b}(\p,\p),
\]
all rows with $b<q$ vanish by minimality and Lemma~\ref{lem:hoch}, whereas
\[
 E_2^{0,q}=H^q(\uu,\p)^{\llie}\neq0.
\]
No differential enters the leftmost column, and every differential leaving $E_r^{0,q}$ lands in a lower
row or outside the first quadrant. The class therefore survives to $E_\infty$, contradicting Tolpygo's
theorem.
\end{proof}
For the seaweed $\s=\p(A\mid B)$ define
\[
 \qq_A^-:=\rr\ltimes\uu_A^-,\qquad
 \qq_B^+:=\rr\ltimes\uu_B^+.
\]
These are parabolic algebras up to central direct factors.

\begin{lemma}\label{lem:onesided}
For every $q>0$,
\[
 H^q(\uu_A^-,\qq_A^-)^{\rr}=0,\qquad
 H^q(\uu_B^+,\qq_B^+)^{\rr}=0.
\]
Moreover,
\[
 H^0(\uu_A^-,\qq_A^-)^{\rr}=\kk_A,\qquad
 H^0(\uu_B^+,\qq_B^+)^{\rr}=\kk_B,
\]
where
\[
 \kk_A=\{h\in\h:\alpha(h)=0\text{ for every }\alpha\in A\},
\]
and similarly for $\kk_B$.
\end{lemma}
\begin{proof}
Let $\g_A$ be the semisimple regular subalgebra with root system $\Delta_A$, and let
$\h_A=\spn\{h_\alpha:\alpha\in A\}$. Then
\[
 \h=\kk_A\oplus\h_A,
\]
and
\[
 \qq_A^-=\kk_A\oplus\p_{A,C}^-,\qquad \rr=\kk_A\oplus\llie_{A,C},
\]
where $\p_{A,C}^-$ is the opposite parabolic of $\g_A$ with Levi factor $\llie_{A,C}$ determined by
$C=A\cap B$, and with nilradical $\uu_A^-$.

For every positive degree, an $\rr$-invariant cochain on $\uu_A^-$ cannot take values in the zero-weight
summand $\kk_A$: the sum of its input weights is a nonzero negative combination of simple roots in $A$.
Thus, in positive degree, the invariant cochain spaces agree with those of the relative complex
\[
 C^\bullet(\uu_A^-,\p_{A,C}^-)^{\llie_{A,C}}.
\]
In degree zero one has
\[
 C^0(\uu_A^-,\qq_A^-)^{\rr}=\kk_A\oplus C^0(\uu_A^-,\p_{A,C}^-)^{\llie_{A,C}},
\]
and the additional summand $\kk_A$ has zero differential. Lemma~\ref{lem:relative} now gives the positive-degree
vanishing and leaves precisely $\kk_A$ in degree zero. The argument for $B$ is identical.
\end{proof}

\subsection{The invariant complex}
Put
\[
 I_-=A\setminus B,\qquad I_+=B\setminus A,\qquad I_0=\Pi\setminus(A\cup B).
\]
The center of $\rr$ is
\[
 Z(\rr)=\bigoplus_{\alpha_i\notin C}\C\varpi_i^\vee
       =V_-\oplus V_+\oplus Z(\s),
\]
where
\[
 V_-:=\bigoplus_{\alpha_i\in I_-}\C\varpi_i^\vee,\qquad
 V_+:=\bigoplus_{\alpha_i\in I_+}\C\varpi_i^\vee.
\]

\begin{proposition}\label{prop:complex}
For every $q\geq1$, restriction to the two pure exterior summands gives an isomorphism
\[
 C^q(\n,\s)^{\rr}\cong C^q(\uu_A^-,\qq_A^-)^{\rr}
                         \oplus C^q(\uu_B^+,\qq_B^+)^{\rr}.
\]
In positive degree the differential is componentwise. In degree zero it is
\[
 d:Z(\rr)\longrightarrow C^1(\uu_A^-,\qq_A^-)^{\rr}
                    \oplus C^1(\uu_B^+,\qq_B^+)^{\rr},
 \qquad x\longmapsto(d_Ax,d_Bx).
\]
\end{proposition}
\begin{proof}
An $\rr$-invariant cochain is $\h$-invariant. If its inputs are root vectors of roots
$\gamma_1,\ldots,\gamma_q$, then its value has weight $\gamma_1+\cdots+\gamma_q$.

If every input lies in $\uu_A^-$ and the total weight is a root, then it lies in
$-(\Delta_A^+\setminus\Delta_C^+)$; hence the value lies in $\uu_A^-$. The analogous assertion holds
on the $B$-side. If inputs from both sides occur, their total weight has a negative coefficient at a
simple root in $A\setminus B$ and a positive coefficient at a simple root in $B\setminus A$.
It is neither zero nor a root. Therefore every invariant cochain vanishes on mixed exterior summands.
Extension by zero on the mixed direct summands gives the inverse restriction map.

Lemma~\ref{lem:commute} shows that cross brackets and cross module actions vanish, so the differential is
componentwise in positive degree. Finally,
\[
 C^0(\n,\s)^{\rr}=\s^{\rr}=Z(\rr),
\]
and the displayed degree-zero differential is the adjoint coboundary restricted to the two one-sided ideals.
\end{proof}

\begin{theorem}\label{thm:invariant}
For every seaweed $\s=\rr\ltimes\n$,
\[
 H^q(\n,\s)^{\rr}=\begin{cases}Z(\s),&q=0,\\0,&q>0.\end{cases}
\]
\end{theorem}
\begin{proof}
By \eqref{eq:invariants}, it is enough to compute the cohomology of the invariant cochain complex.
For $q\geq2$, Proposition~\ref{prop:complex} identifies a cocycle with a pair of one-sided cocycles.
Lemma~\ref{lem:onesided} makes both components coboundaries. Their primitives have degree at least one and therefore
combine, under the same splitting, to a primitive in the seaweed complex.

For $q=1$, let $(f_-,f_+)$ be a cocycle. By Lemma~\ref{lem:onesided}, there exist $x_-,x_+\in Z(\rr)$ such that
\[
 f_-=d_Ax_-,\qquad f_+=d_Bx_+.
\]
The map $d_A$ annihilates $V_+\oplus Z(\s)$ and is injective on $V_-$, since $\uu_A^-$
contains $\g_{-\alpha_i}$ for every $\alpha_i\in I_-$. Similarly, $d_B$ annihilates
$V_-\oplus Z(\s)$ and is injective on $V_+$. If $x_-^{(-)}$ and $x_+^{(+)}$ denote the respective
projections, then
\[
 x=x_-^{(-)}+x_+^{(+)}
\]
satisfies $dx=(f_-,f_+)$. Thus the degree-one cohomology vanishes.

In degree zero, the kernel consists of the elements centralizing both $\rr$ and $\n$, hence is $Z(\s)$.
\end{proof}

\section{Central comparison and ordinary cohomology}
We first compute ordinary cohomology. Choose $d\in\h$ by
\[
 \alpha_i(d)=\begin{cases}
 -1,&\alpha_i\in A\setminus B,\\
 1,&\alpha_i\in B\setminus A,\\
 0,&\text{otherwise}.
 \end{cases}
\]
Every root occurring in $\n$ has strictly positive $d$-weight. Consequently, every positive-degree scalar
cochain on $\n$ has strictly negative $d$-weight.

\begin{proposition}\label{prop:ordinary}
Inflation along the quotient map $\s\twoheadrightarrow\s/\n\cong\rr$ induces an isomorphism of graded algebras
\[
 H^\bullet(\rr,\C)\xrightarrow{\ \sim\ }H^\bullet(\s,\C).
\]
\end{proposition}
\begin{proof}
Since taking $\rr$-invariants is exact, the weight observation gives
\[
 H^q(\n,\C)^{\rr}=0\qquad(q>0).
\]
In the Hochschild--Serre spectral sequence
\[
 E_2^{p,q}=H^p\bigl(\rr,H^q(\n,\C)\bigr)\Longrightarrow H^{p+q}(\s,\C),
\]
Lemma~\ref{lem:hoch} annihilates every positive row. The bottom row is $H^p(\rr,\C)$. The spectral sequence
therefore collapses, and its edge map is inflation. Multiplicativity gives the algebra statement.
\end{proof}
We now compare central and adjoint coefficients.

\begin{proof}[Proof of Theorem~\ref{thm:main}]
The center $Z=Z(\s)$ is a trivial $\s$-module, and the inclusion $Z\hookrightarrow\s$ is
a morphism of $\s$-modules. It therefore induces a morphism between the two Hochschild--Serre
spectral sequences associated with the ideal $\n$.

For central coefficients,
\[
 H^q(\n,Z)\cong H^q(\n,\C)\otimes Z.
\]
The grading-element weight argument used in the proof of Proposition~\ref{prop:ordinary} gives
\[
 H^q(\n,Z)^{\rr}=\begin{cases}Z,&q=0,\\0,&q>0.\end{cases}
\]
For adjoint coefficients, Theorem~\ref{thm:invariant} gives the same invariant spaces. Since all the relevant modules
are semisimple, Lemma~\ref{lem:hoch} shows that both $E_2$-pages vanish off the bottom row and that their
bottom rows are
\[
 H^p(\rr,\C)\otimes Z.
\]
By naturality of Hochschild's isomorphism in the coefficient module, the map induced by
$Z\hookrightarrow\s$ is an isomorphism on this bottom row. The comparison theorem for spectral
sequences therefore gives
\[
 H^n(\s,Z)\xrightarrow{\ \sim\ }H^n(\s,\s)
\]
for every $n$. This proves part~(i). Part~(ii) is Proposition~\ref{prop:ordinary}. Finally,
\[
 H^\bullet(\s,Z)\cong H^\bullet(\s,\C)\otimes Z
\]
because $Z$ is trivial, and parts~(i) and~(ii) give part~(iii).
\end{proof}

\begin{remark}\label{rem:map}
Theorem~\ref{thm:main}(i) is stronger than an abstract equality of cohomology dimensions: it
identifies the actual coefficient inclusion as a quasi-isomorphism of Chevalley--Eilenberg complexes.
\end{remark}

\section{Consequences}
\subsection{Adjoint acyclicity, completeness, and Frobenius seaweeds}
Recall that a Lie algebra is complete if its center is zero and every derivation is inner.

\begin{corollary}\label{cor:complete}
For a seaweed $\s=\p(A\mid B)$, the following conditions are equivalent:
\begin{enumerate}[label=(\roman*),leftmargin=*,itemsep=0pt]
\item $A\cup B=\Pi$;
\item $Z(\s)=0$;
\item $H^\bullet(\s,\s)=0$;
\item $\s$ is complete.
\end{enumerate}
\end{corollary}
\begin{proof}
The equivalence of (i) and (ii) is Theorem~\ref{thm:center}. Theorem~\ref{thm:main} gives the equivalence of (ii)
and (iii). If these conditions hold, then $H^1(\s,\s)=0$, so every derivation is inner and $\s$ is complete.
Conversely, completeness includes the condition $Z(\s)=0$.
\end{proof}

\begin{corollary}\label{cor:frob}
Every Frobenius seaweed is complete and adjoint-acyclic.
\end{corollary}
\begin{proof}
The center of any Lie algebra lies in every coadjoint stabilizer. A Frobenius Lie algebra has
a zero-dimensional coadjoint stabilizer, so its center is zero. Apply Corollary~\ref{cor:complete}.
\end{proof}
In particular, every indecomposable seaweed has $H^2(\s,\s)=0$ and is formally rigid by the
classical Nijenhuis--Richardson criterion \cite{NR67}. The converse statement with only $H^2$ is false;
decomposable seaweeds can also have vanishing second adjoint cohomology.

\subsection{Low-degree groups and all adjoint Betti numbers}
Let
\[
 \cc=Z(\rr),\qquad c=\dim\cc,\qquad z=\dim Z(\s).
\]
Write
\[
 [\rr,\rr]=\rr_1\oplus\cdots\oplus\rr_m
\]
with the $\rr_a$ simple. If $m_{a,1},\ldots,m_{a,\ell_a}$ are the exponents of $\rr_a$, the classical
Chevalley--Eilenberg description of reductive cohomology gives \cite{CE48}
\begin{equation}\label{eq:reductive}
 H^\bullet(\rr,\C)\cong\bigwedge^\bullet\cc^*\otimes
 \bigotimes_{a=1}^m\bigwedge\bigl(\eta_{a,1},\ldots,\eta_{a,\ell_a}\bigr),
 \qquad \deg\eta_{a,j}=2m_{a,j}+1.
\end{equation}

\begin{corollary}\label{cor:low}
There are canonical vector-space isomorphisms
\begin{align*}
 H^0(\s,\s)&\cong Z(\s),\\
 H^1(\s,\s)&\cong\cc^*\otimes Z(\s),\\
 H^2(\s,\s)&\cong\bigwedge^2\cc^*\otimes Z(\s),\\
 H^3(\s,\s)&\cong\left(\bigwedge^3\cc^*\oplus\C^m\right)\otimes Z(\s).
\end{align*}
In particular,
\[
 \Out(\s)=\Der(\s)/\ad(\s)\cong\Hom(\cc,Z(\s)).
\]
\end{corollary}
\begin{proof}
Combine Theorem~\ref{thm:main} with \eqref{eq:reductive}. Each simple factor has one exponent equal to $1$, and no
generator in degrees $1$ or $2$. For the degree-three summand, the copy of $\C$ attached
to $\rr_a$ is identified using the cocycle $(x,y,z)\mapsto\kappa_a(x,[y,z])$, where
$\kappa_a$ is the Killing form of $\rr_a$.
\end{proof}
Since $\rr$ is determined by $C=A\cap B$ and its center is the annihilator of $C$ in $\h$, one has
\[
 c=\ell-|A\cap B|,\qquad z=\ell-|A\cup B|.
\]
Thus Corollary~\ref{cor:low} gives the explicit dimensions
\[
 \dim H^1(\s,\s)=zc,\qquad
 \dim H^2(\s,\s)=z\binom c2,\qquad
 \dim H^3(\s,\s)=z\left(\binom c3+m\right).
\]
Here $m$ is the number of connected components of the Dynkin subdiagram on $A\cap B$.

\begin{corollary}[Infinitesimal rigidity criterion]\label{cor:rigidity}
A seaweed satisfies $H^2(\s,\s)=0$ if and only if
\[
 Z(\s)=0\quad\text{or}\quad\dim Z(\rr)\leq1.
\]
Equivalently,
\[
 A\cup B=\Pi\quad\text{or}\quad |A\cap B|\geq\ell-1.
\]
\end{corollary}
\begin{proof}
This is immediate from $H^2(\s,\s)\cong\bigwedge^2\cc^*\otimes Z(\s)$.
\end{proof}
Define the adjoint Poincar\'e polynomial by
\[
 P^{\mathrm{ad}}_{\s}(t)=\sum_{n\geq0}\dim H^n(\s,\s)t^n.
\]

\begin{corollary}[Poincar\'e polynomial]\label{cor:poincare}
With the notation above,
\[
 P^{\mathrm{ad}}_{\s}(t)=z(1+t)^c\prod_{a=1}^m\prod_{j=1}^{\ell_a}
 \bigl(1+t^{2m_{a,j}+1}\bigr).
\]
Consequently, all adjoint Betti numbers are determined by the union $A\cup B$, the intersection $A\cap B$,
and the exponents of the connected components of the Dynkin subdiagram on $A\cap B$.
\end{corollary}
\begin{proof}
Take dimensions in Theorem~\ref{thm:main} and use \eqref{eq:reductive}.
\end{proof}
The polynomial is palindromic of degree $\dim\rr$ whenever it is nonzero. Hence, for
$0\leq n\leq\dim\rr$,
\[
 \dim H^n(\s,\s)=\dim H^{\dim\rr-n}(\s,\s).
\]
Moreover, nonzero adjoint cohomology is concentrated in degrees at most $\dim\rr$, even when the
ideal $\n$ is large.

\subsection{The centerless quotient and decomposable seaweeds}
Let
\[
 Z=Z(\s),\qquad Q=\s/Z.
\]
Proposition~\ref{prop:split} gives a root-theoretic Lie algebra splitting $\s=Z\oplus Q$ after identifying
$Q$ with the complement $\s_S$.

\begin{proposition}\label{prop:quotient}
The quotient $Q$ is adjoint-acyclic:
\[
 H^\bullet(Q,Q)=0.
\]
For every $j\geq0$,
\[
 H^j(Q,\s)\cong H^j(Q,\C)\otimes Z.
\]
\end{proposition}
\begin{proof}
By Proposition~\ref{prop:components}, $Q=\bigoplus_a\s_a$ with every $\s_a$ indecomposable.
Corollary~\ref{cor:complete} gives $H^\bullet(\s_a,\s_a)=0$. The K\"unneth formula for a direct sum,
applied to each coefficient summand $\s_a\subseteq Q$, gives $H^\bullet(Q,Q)=0$.
As a $Q$-module, $\s=Z\oplus Q$ with trivial action on $Z$, so
\[
 H^j(Q,\s)\cong H^j(Q,Z)\oplus H^j(Q,Q)\cong H^j(Q,\C)\otimes Z.
\]
\end{proof}

\begin{corollary}[Central-product formula]\label{cor:product}
For every $n\geq0$,
\[
 H^n(\s,\s)\cong\bigoplus_{i+j=n}
 \left(\bigwedge^i Z^*\otimes H^j(Q,\C)\otimes Z\right).
\]
The direct-sum realization is natural after the root-theoretic central complement $Q\hookrightarrow\s$ has been
fixed.
\end{corollary}
\begin{proof}
Since $Z$ is central and acts trivially on $\s$, the exterior splitting of $Z\oplus Q$ gives an isomorphism
of cochain complexes
\[
 C^\bullet(\s,\s)\cong\bigwedge^\bullet Z^*\otimes C^\bullet(Q,\s),
\]
with the usual tensor-product differential, acting on a homogeneous element
$\alpha\otimes f$ with $\alpha\in\bigwedge^i Z^*$ as
\[
 d(\alpha\otimes f)=(-1)^i\alpha\otimes d_Qf.
\]
Take cohomology and apply Proposition~\ref{prop:quotient}.
\end{proof}

\begin{remark}\label{rem:CG}
Corollary~\ref{cor:product} is the zero-action case of the toral semidirect-product theorem of
Coll and Gerstenhaber \cite{CG16}. The direct cochain factorization here makes the central case
independent of any semidirect-product hypotheses.
\end{remark}

\subsection{The deformation differential graded Lie algebra}
Use the full shifted Chevalley--Eilenberg complex
\[
 L_{\s}^{\,k}:=C^{k+1}(\s,\s),\qquad k\geq-1,
\]
with the Nijenhuis--Richardson bracket and the compatible shifted Chevalley--Eilenberg differential.
In particular, $L_{\s}^{-1}=C^0(\s,\s)=\s$ is retained. This differential graded Lie algebra
controls formal deformations of the Lie bracket \cite{NR67}. The central-valued cochains
\[
 L_Z^{\,k}:=C^{k+1}(\s,Z),\qquad k\geq-1,
\]
form a differential graded Lie subalgebra: insertion of central-valued cochains remains central-valued,
and the differential preserves central coefficients.

\begin{corollary}\label{cor:dgla}
The inclusion
\[
 L_Z\hookrightarrow L_{\s}
\]
is a quasi-isomorphism of differential graded Lie algebras. Thus the full shifted deformation complex
of a seaweed admits a quasi-isomorphic model consisting entirely of central-valued cochains.
\end{corollary}
\begin{proof}
The differential graded Lie algebra inclusion is induced by the coefficient inclusion $Z\hookrightarrow\s$.
For every $k\geq-1$, its map on cohomology is the isomorphism
$H^{k+1}(\s,Z)\longrightarrow H^{k+1}(\s,\s)$ of Theorem~\ref{thm:main}(i).
\end{proof}
The degree $-1$ terms are essential here: deleting them by nonnegative truncation generally
changes degree-zero cohomology and does not preserve this quasi-isomorphism statement.
Central-valued cochains need not have zero Nijenhuis--Richardson bracket, so the corollary
makes no general unobstructedness assertion.

\section{Examples}
\begin{example}[Parabolic and Levi extremes]\label{ex:extremes}
A standard parabolic is $\p_A=\p(A\mid\Pi)$. Since
$A\cup\Pi=\Pi$, Theorem~\ref{thm:center} gives $Z(\p_A)=0$, and Theorem~\ref{thm:main} recovers
\[
 H^\bullet(\p_A,\p_A)=0.
\]
At the other extreme, $\p(A\mid A)$ is the standard Levi subalgebra $\llie_A$. Here $\rr=\s=\llie_A$, and the main
theorem becomes
\[
 H^\bullet(\llie_A,\llie_A)\cong H^\bullet(\llie_A,\C)\otimes Z(\llie_A).
\]
Thus the seaweed formula interpolates exactly between the classical parabolic and reductive
formulas.
\end{example}

\begin{example}[A decomposable seaweed with $H^2=0$]\label{ex:cartan}
Let $\g=\mathfrak{sl}_2$ and take $A=B=\varnothing$. Then
$\s=\h$ is one-dimensional abelian. It is decomposable in the seaweed sense and is not complete, but
\[
 H^0(\s,\s)\cong\s,\qquad H^1(\s,\s)\cong\s,\qquad
 H^q(\s,\s)=0\quad(q\geq2).
\]
This illustrates why indecomposability is equivalent to all-degree adjoint acyclicity, but not to the
isolated condition $H^2=0$. In this case $\n=0$, whereas $\Nil(\s)=Z(\s)=\s$.
\end{example}

\begin{example}[A three-dimensional $G_2$ seaweed]\label{ex:g2}
Let $A=\{\alpha_1\}$ and $B=\varnothing$ in type $G_2$. The
seaweed $\p(A\mid B)$ is isomorphic to the Lie algebra with basis $\{h,x,z\}$ and brackets
\[
 [h,x]=x,\qquad [h,z]=[x,z]=0.
\]
Thus $Z(\s)=\C z$ and $Q=\s/Z(\s)$ is the two-dimensional nonabelian Lie algebra. Its trivial-coefficient
cohomology is
\[
 H^0(Q,\C)=\C,\qquad H^1(Q,\C)=\C,\qquad H^j(Q,\C)=0\quad(j\geq2).
\]
Corollary~\ref{cor:product} yields
\[
 \dim H^0(\s,\s)=1,\qquad\dim H^1(\s,\s)=2,\qquad
 \dim H^2(\s,\s)=1,\qquad H^j(\s,\s)=0\quad(j\geq3).
\]
A generator of $H^2$ is represented by the central-valued cocycle
\[
 \phi(h,z)=z,\qquad \phi(h,x)=\phi(x,z)=0.
\]
It integrates without higher terms to
\[
 [h,x]_t=x,\qquad [h,z]_t=tz,\qquad [x,z]_t=0.
\]
For every $t$, the displayed bracket satisfies the Jacobi identity, so this infinitesimal class is
unobstructed and integrates without higher-order correction terms.
\end{example}

\section*{Acknowledgments}
The authors thank Nick Mayers for helpful discussions.

\section*{Statements and Declarations}
\noindent\textbf{Funding.} The authors received no funding for this work.

\noindent\textbf{Competing interests.} The authors have no relevant financial or non-financial interests to disclose.

\end{document}